\documentclass[11pt]{article}
\usepackage{amssymb,amsmath,amsfonts}
\usepackage{tikz}
\usepackage{graphics}
\usepackage{graphicx}

\DeclareMathOperator{\real}{Re}

\def\eqref#1{$(\ref{#1})$}

\newenvironment{proof}{\noindent {\em {Proof}}.}{$\square$
\medskip}
\newtheorem{theorem}{Theorem}[section]
\newtheorem{corollary}[theorem]{Corollary}
\newtheorem{lemma}[theorem]{Lemma}
\newtheorem{remark}[theorem]{Remark}
\newtheorem{proposition}[theorem]{Proposition}

\newtheorem{example}[theorem]{Example}
\newtheorem{question}[theorem]{Question}
\newtheorem{problem}[theorem]{Problem}

\newcommand{\disk}{\mathbb{D}}

\newcommand{\bt}[1]{\begin{theorem}\label{#1}}
\newcommand{\bc}[1]{\begin{corollary}\label{#1}}
\newcommand{\bl}[1]{\begin{lemma}\label{#1}}
\newcommand{\bp}[1]{\begin{proposition}\label{#1}}
\newcommand{\be}[1]{\begin{example}\rm\label{#1}}
\newcommand{\bq}[1]{\begin{question}\rm\label{#1}}
\newcommand{\bprob}[1]{\begin{problem}\rm\label{#1}}
\newcommand{\beq}[1]{\begin{eqnarray}\label{#1}}
\newcommand{\br}[1]{\begin{remark}\rm\label{#1}}

\newcommand{\el}{\end{lemma}}
\newcommand{\ep}{\end{proposition}}
\newcommand{\ee}{\end{example}}
\newcommand{\eq}{\end{question}}
\newcommand{\eprob}{\end{problem}}
\newcommand{\eeq}{\end{eqnarray}}
\newcommand{\ed}{\end{definition}}
\newcommand{\et}{\end{theorem}}
\newcommand{\ec}{\end{corollary}}
\newcommand{\er}{\end{remark}}

\title{\bf Non-radial weights and polynomial approximation in spaces of analytic functions}
\author{Ali Abkar
\vspace{0.5in}\\
Department of Pure Mathematics, Faculty of Science,\\
Imam Khomeini International University, Qazvin 34149, Iran\\
\\
\\
\texttt{Email:~abkar@sci.ikiu.ac.ir}\\
 }
\date{}
\begin{document}
\maketitle \textbf{Abstract.} We study sufficient conditions on weight functions under which norm
 approximations by analytic polynomials are possible. The weights we study include radial, non-radial, and angular weights.\\


\textbf{}{}

\section{Introduction}

Let $X$ be a Banach space of analytic functions on some bounded domain $G$ in the complex plane. It is an important question to know if $X$ contains the polynomials; and if the answer is affirmative, then the question pops up if the polynomials are dense in $X$. In classical complex analysis, if $K$ is a compact set whose complement is connected, and if $f$ is continuous on $K$ and analytic in the interior of $K$, then $f$ can be uniformly approximated by polynomials on $K$ (Mergelyan's Theorem, \cite{rud}, p. 390). But in spaces of analytic functions, the continuity assumption on the boundary of the domain, as well as the compactness assumptions are missing, so that the problem needs a new treatment in every concrete space. \par
Let $0<p<\infty$ and $G$ be a bounded simply connected domain that is conformally equivalent to the unit disk $\disk=\{z: |z|<1\}$. The Bergman space
$A^p(G)$ consists of analytic functions in $G$ such that
$$\|f\|^p_{A^p(G)}=\int_G|f(z)|^pdA(z)$$
is finite; here $dA(z)=dxdy$ is the usual planar measure on $G$. It is well-known that for $1\le p<\infty$ the space $A^p(G)$
is a Banach space; and for $0<p<1$ it is a complete metric space. \par
Putting aside the uniform approximation of functions by polynomials,
the first result on the approximation by polynomials in the spaces of square ($p$-th power) area integrable functions was recorded by Torsten Carleman in his 1923 paper who proved the result for Jordan regions \cite{car}. At the time the phrase "Bergman space" was not yet coined!
This result was then extended by O. J. Farrell to Carath\'{e}odory regions (see \cite{far1}, \cite{far2}).
Farrell proved that for each $f\in A^p(G)$, there is a sequence of polynomials $p_n$ such that $p_n\to f$ uniformly on compact subsets of $G$, and
in the mean (or in norm):
$$\|f-p_n\|_{A^p(G)}\to 0,\quad n\to\infty.$$
Here, we should remark that the Taylor polynomials of a given function do not necessarily converge to the same function in norm; this statement is true if $1<p<\infty$ (see \cite{pzz}, and \cite{zhu}).\par
In this paper we are mainly concerned with approximation in weighted Bergman spaces.
A weight function is a positive integrable function $w:G\to \mathbb{R}$. The weighted Bergman space $A^p(G, wdA)=A_w^p(G)$ consists of analytic functions in $G$ such that
$$\int_G|f(z)|^pw(z)dA(z)<\infty.$$
 It is well-known that if the weight function is radial, that is, if $w(z)=w(|z|)$, then the polynomials are dense in the weighted Bergman space $A^p(G, wdA)$. This result is attributed to Sergey N. Mergelyan who proved that for $p=2$, and $G=\disk$, the dilatations $f_r(z)=f(rz)$ converge to $f$ in norm as $r\to 1^-$ (see \cite{mer}). This entails the density of polynomials in the weighted Bergman space.
If the weight function is not radial, the following few simple cases are known to behave well; that is, the associated spaces allow the polynomial approximation.
For example,\\
1) if $w$ is bounded in $\disk$, then we have
$$\int_\disk |f(z)|^pw(z)dA(z)\le C\int_\disk |f(z)|^pdA(z)$$
and since the polynomials are dense in $A^p(\disk,dA)$, it follows that they are dense in $A^p(\disk,wdA)$.\\
2) if there is a constant $C$ such that (for $d\mu(z)=w(z)dA(z)$)
$$\int_\disk |f(z)|^pd\mu(z)\le C\int_\disk |f(z)|^pdA(z),$$
then the polynomials are dense in $A^p(\disk,d\mu)$. Such measures are called Carleson measures for the Bergman space.\\
Approximation in general weighted Bergman spaces is much more complicated and remains still unsolved. Apart from the pioneering work of Mergelyan, and Farrell, credits should also be given to Lars Inge Hedberg for his investigation on weighted approximation by polynomials on compact subsets of the complex plane (see \cite{hed1}, and \cite{hed2}). \par
In recent years, the current author considered a class of non-radial weights that are super-biharmonic (their bi-laplacian is non-negative) and satisfy the growth condition
\begin{equation}\label{cond-weight}
\lim_{r\to 1}\int_\mathbb{T} w(rz)d\sigma(z)=0,
\end{equation}
on the boundary of the unit disk $\mathbb{T}=\partial \disk$ (see \cite{abkar1}, \cite{abkar2}). We proved that for this class of weights the polynomial approximation is possible in Bergman spaces, and provided some applications in the invariant subspaces of the Bergman space. The same result also holds for the Dirichlet spaces, and the analytic Besov spaces.\par
In this paper we prove that if the weight $w$ satisfies the condition
$$r^kw\left (\frac{z}{r}\right )\le Cw(z),\quad |z|<r,\, r_0\le r<1,\, r_0\in (0,1),$$
where $C$ is some constant and $k$ is some non-negative integer, then the polynomials are dense in weighted Bergman space $A^p(\disk, wdA)$.
The motivation comes from our earlier work to show that for super-biharmonic weights satisfying the condition (\ref{cond-weight}), the function
$$r\mapsto rw(\frac{z}{r})$$
is increasing in $r$ if $|z|<r$ (see \cite{abkar1}). This suggests that to achieve the polynomial approximation, one has to consider the behavior of the function $rw(z/r)$. The results we obtain, extend our previous results on this topic.\par
Next we study the angular weights on the unit disk. These are weights that
depend just on $\theta$; in contrast to radial weights that depend only on $r=|z|$.
For instance, we consider
$$w(z)=\left (1-\frac{\theta}{2\pi}\right ),\quad z=re^{i\theta}.$$
This is an analog for the radial weight $w(z)=1-|z|$ in the unit disk.\\
We shall then investigate weights of the form
$$w(z)=\omega(r)v(\theta),$$
where $\omega$ is a radial weight and $v$ is an angular weight. We shall see that in all this cases, the polynomial approximation is possible.
\par
Finally, in the last section we discuss the similar problem in general domains (see \S 3).
We close this section by mentioning that polynomial approximation is closely connected to the invariant subspaces of the space in question (see \cite{abkar1} and \cite{abkar2}); as well as to the behavior of composition operators defined on the underlying domain, see for instance \cite{chz}.
This issue also pops up in the decomposition of functions in the weighted Bergman spaces; see a recent paper of Korhonen and R\"{a}tty\"{a} where the authors assume that the weight is chosen in such a way that the polynomials are dense in the weighted Bergman space $A^p_w$ (\cite{kor}, Theorem 4).


\section{Approximation in the unit disk}
It is a classical theorem that if the weight function on the unit disk is radial, that is $w(z)=w(|z|)$, then the polynomials are dense in the weighted Bergman space $A^2(\disk, wdA)$. This was proved by S. Mergelyan \cite{mer} under the assumption that
$$\int_0^1rw(r)dr<\infty.$$
The same result is of course true for the Bergman spaces $A^p(\disk, wdA)$, for $0<p<\infty$.

We assume now that $w$ is an arbitrary weight function (not necessarily radial) for which
$$\int_\disk w(z)dA(z)<\infty.$$
We shall see that under some mild condition on $w$, the polynomials are dense in the weighted Bergman spaces $A^p(\disk, wdA)$.
We begin with the following theorem.
\begin{theorem}\label{limsup}
Let $0<p<\infty$ and $d\mu(z)=w(z)dA(z)$ be a finite positive measure on the unit disk such that for some non-negative
integer $k$ and an $r_0\in (0,1)$ we have
\begin{equation}\label{con-main}
r^kw\left (\frac{z}{r}\right )\le Cw(z),\quad |z|<r,\, r_0\le r<1.
\end{equation}
Then the polynomials are dense in the weighted Bergman space $A^p(\disk,d\mu)$.
\end{theorem}
\begin{proof}
We show that for each $\epsilon>0$, there is a polynomial $Q$ such that $\|f-Q\|_{A^p(\disk,d\mu)}<\epsilon.$
Let $f\in A^p(\disk, d\mu)$, and consider the dilatations $f_r(z)=f(rz)$ for $0\le r<1$, and $z\in\disk$.
By a change of variables, we see that for $r$ sufficiently close to $1$, we have
\begin{align*}\int_\disk |f_r(z)|^pd\mu(z)&=
\frac{1}{r^{2+k}}\int_{r\disk} |f(z)|^p r^k w\left(\frac{z}{r} \right)dA(z)\\
&\le \frac{C}{r^{2+k}}\int_{r\disk} |f(z)|^p w(z)dA(z).
\end{align*}
We now use the dominated convergence theorem to conclude that
$$\limsup_{r\to 1^{-}} \int_\mathbb{D}|f_r|^pd\mu(z)\le \int_\mathbb{D}|f|^pd\mu(z).$$
This inequality together with Fatou's lemma implies that for each measurable subset $E\subseteq \disk$,
\begin{equation}\label{eq-lim-E}
\lim_{r\to 1^{-}} \int_{E}|f_r|^pd\mu(z)= \int_{E}|f|^pd\mu(z).
\end{equation}
By the continuity property of integral on its domain, there is a $\delta>0$ such that
$$\mu(E)<\delta \implies \int_{E}|f|^pd\mu(z)<\epsilon.$$
According to Egorov's theorem, there is a subset $D_0$ of the unit disk such that $\mu(\disk-D_0)<\delta$ and $f_r\to f$ uniformly on $D_0$.
We now write
\begin{align}\label{2max}
\int_\disk |f_r-f|^p d\mu &= \int_{D_0} |f_r-f|^p d\mu +\int_{\disk\setminus D_0} |f_r-f|^p d\mu  \nonumber \\
&\le \int_{D_0} |f_r-f|^p d\mu +2^p\int_{\disk\setminus D_0} (|f_r|^p+|f|^p) d\mu.
\end{align}
Note that due to the uniform convergence on $D_0$, the first integral on the right-hand side of (\ref{2max}) tends to zero as $r\to 1^-$.
As for the second integral, we use (\ref{eq-lim-E}) to obtain
$$\limsup_{r\to 1^{-}} \int_\mathbb{D}|f_r-f|^pd\mu(z)\le 2^{p+1}\int_{\disk\setminus D_0} |f|^p d\mu \le 2^{p+1}\epsilon.$$
This latter inequality entails
\begin{equation}\label{ffrto0}
\lim_{r\to 1^{-}}\int_\mathbb{D}|f_r-f|^pd\mu(z)=0.
\end{equation}
Given $\epsilon>0$, there is $0<r_0<1$ such that $\|f-f_{r_0}\|_{A^p(\disk,d\mu)}<\epsilon$.
Since $f_{r_0}$ is analytic in a disk bigger than the unit disk, we may use Mergelyan's theorem to approximate $f_{r_0}$ uniformly on $\disk$ by a polynomial $Q$. Since $\mu$ is a finite positive measure on the unit disk, we obtain
$$\| f_{r_0}-Q\|^p_{A^p(\disk,d\mu)}=\int_\disk |f_{r_0}-Q|^p d\mu \le \epsilon^p\int_\disk d\mu(z):=\epsilon^pC,$$
from which it follows that
$$\| f-Q\|_{A^p(\disk,d\mu)}\le \epsilon(1+C^{1/p}).$$
This proves the theorem.
\end{proof}

\begin{remark}
The condition (\ref{con-main}) can be weakened in the following way: if there is a non-negative function $g$ that dominates $r^k w(z/r)$ and such that
$$\int_\disk |f(z)|^p g(z)dA(z)<\infty.$$
Again the dominated convergence theorem can be applied to deduce the key inequality
\begin{equation*}
\limsup_{r\to 1^{-}} \int_{E}|f_r|^pd\mu(z)\le \int_{E}|f|^pd\mu(z).
\end{equation*}
\end{remark}

\begin{theorem}\label{increasing}
Let $0<p<\infty$ and $d\mu(z)=w(z)dA(z)$ be a finite positive measure on the unit disk such that for some non-negative integer $k$ and each $z$ with $|z|<r$, the function
$$r\mapsto r^kw\left (\frac{z}{r}\right )$$
is increasing. Then the polynomials are dense in the weighted Bergman space $A^p(\disk,wdA)$.
\end{theorem}
\begin{proof}
We will verify that each function $f\in A^p(\disk,d\mu)$ satisfies the condition of Theorem \ref{limsup}. By a change of variables, it is easy to see that
\begin{equation*}\int_\disk |f_r(z)|^pd\mu(z)=
\frac{1}{r^{2+k}}\int_{r\disk} |f(z)|^p r^k w\left(\frac{z}{r} \right)dA(z).
\end{equation*}
We now use the monotone convergence theorem to conclude that
$$\limsup_{r\to 1^-}\int_\disk |f_r(z)|^pd\mu(z)=\int_\disk |f(z)|^pd\mu(z),$$
from which the result follows.
\end{proof}

\begin{example} (a).~
Let $w(z)=(1-|z|^2)^\alpha$ where $\alpha>0$. We then have
$$\int_\disk |f_r(z)|^pd\mu(z)=\int_\disk |f_r(z)|^p (1-|z|^2)^\alpha dA(z).$$
Substitute $z$ by $z/r$ in the right to get (for $z\in r\disk$):
$$\int_\disk |f_r(z)|^pd\mu(z)=\frac{1}{r^2}\int_{r\disk} |f(z)|^p \left(\frac{r^2-|z|^2}{r^2}\right)^\alpha dA(z).$$
Note that $$r\mapsto \frac{r^2-|z|^2}{r^2}$$
is increasing in $r$, so that by the monotone convergence theorem
\begin{align*}\limsup_{r\to 1^-}\int_\disk |f_r(z)|^pd\mu(z)&=\lim_{r\to 1^-}
\frac{1}{r^2}\int_{r\disk} |f(z)|^p \left(\frac{r^2-|z|^2}{r^2}\right)^\alpha dA(z)\\
&=\int_\disk |f(z)|^pd\mu(z).
\end{align*}
(b).~ Note also that for $\alpha>-1$, and the standard weight $w(z)=(\alpha+1)(1-|z|^2)^\alpha$ we have
$$r^kw(\frac{z}{r})=(\alpha+1)r^{k-2\alpha}(r^2-|z|^2)^\alpha\le w(z),$$
provided that $k-2\alpha >1$. Now Theorem \ref{limsup} can be applied.\\
(c).~ Consider the non-radial weight $w(z)=|\real (z)|=|x|$. We have $rw(z/r)=r|x/r|=|x|=w(z)$, so that the above Theorems apply.\\
(d).~
It is easily seen that the Gaussian weight $w(z)=\exp(-|z|^2)$ satisfies the condition of the above theorem for $k=0$. The same is true for non-radial
weight $w(z)=\exp(-|\real (z)|^2)$.
We should further remark that in some instances the function $r\mapsto w(z/r)$ is not increasing, but we may find some $k$ for which $r\mapsto r^kw(z/r)$
is increasing. For example, if we take $w(z)=\exp(|z|)$, then
$$\frac{d}{dr}w\left (\frac{z}{r}\right )=-\frac{|z|}{r^2}\exp\left (\frac{|z|}{r}\right )<0$$
while
$$\frac{d}{dr}\left[rw(\frac{z}{r})\right]=\left (1-\frac{|z|}{r}\right )\exp\left (\frac{|z|}{r}\right )>0,\quad |z|<r.$$
Note also that $r^2e^{|z|/r}$ is increasing since
$$\frac{d}{dr}(r^2e^{|z|/r})=2re^{|z|/r}+r^2(\frac{-1}{r^2}e^{|z|/r})=(2r-1)e^{|z|/r}>0$$
for $r>1/2$.
\end{example}

Recall that a weight function $w$ on the unit disk
is said to be super-biharmonic if $\Delta^2 w(z)\ge 0$ where $\Delta$ stands for the laplacian
$$\Delta =\frac{1}{4}\left (\frac{\partial^2}{\partial x^2}+\frac{\partial^2}{\partial y^2}\right).$$

We have already proved that every super-biharmonic weight function satisfying
\begin{equation*}\label{condition2}\lim_{r\to 1}\int_\mathbb{T} w(rz)d\sigma(z)=0,\end{equation*}
has the property that $r\mapsto rw(z/r)$ is increasing in $r$ if $|z|<r$ (see \cite{abkar1} and \cite{abkar2}). This shows that the class of weight functions satisfying the condition (\ref {con-main}) contains the class of super-biharmonic weight functions satisfying the condition
(\ref{cond-weight}). Therefore, Theorem \ref{increasing} extends some of our earlier results announced in \cite{abkar1} and \cite{abkar2}.
\begin{corollary}[\cite{abkar1}] Let $w$ be a super-biharmonic weight function satisfying
\begin{equation*}\label{condition2}\lim_{r\to 1}\int_\mathbb{T} w(rz)d\sigma(z)=0,\end{equation*}
where $d\sigma$ is the normalised arc-length measure on $\mathbb{T}=\partial \disk$.
Then the polynomials are dense in $A^p(\disk, wdA)$.
\end{corollary}

\noindent{\bf Dirichlet and Besov spaces}\\
The results of the preceding section can be modified to incorporate the Dirichlet spaces, as well as the analytic Besov spaces.
The weighted Dirichlet space consists of analytic functions in the unit disk for which
$$\int_{\disk}|f^\prime(z)|^2w(z) dA(z)<\infty.$$
The norm of a function in the weighted Dirichlet space is given by
$$\|f\|_{\mathcal{D}^2_w}=\left( |f(0)|^2+\int_{\disk}|f^\prime(z)|^2w(z) dA(z)\right )^{1/2}.$$
We may also consider the weighted Dirichlet-type spaces $\mathcal{D}^p_w,\, 1<p<\infty$, consisting of analytic functions in the unit disk for which
$$\|f\|_{\mathcal{D}^p_w}=\left( |f(0)|^p+\int_{\disk}|f^\prime(z)|^pw(z)dA(z)\right )^{1/p}$$
is finite.
The weighted Dirichlet-type spaces are particular cases of a scale of Banach spaces of analytic functions, namely, the weighted analytic Besov spaces. Indeed, the weighted analytic Besov space $\mathcal{B}_w^p$ consists of analytic functions $f$ in the unit disk for which the integral
$$\int_{\disk}(1-|z|^2)^{p-2}|f^\prime(z)|^p w(z)dA(z)$$
is finite.
The norm of a function in the weighted Besov space is given by
\begin{equation*}
\| f\|_{\mathcal{B}^p_w}=\left (|f(0)|^p+\int_{\disk}(1-|z|^2)^{p-2}|f^\prime(z)|^p w(z) dA(z)\right )^{1/p}.
\end{equation*}
\begin{theorem}\label{limsup-dirichlet}
Let $0<p<\infty$ and $d\mu(z)=w(z)dA(z)$ be a finite positive measure on the unit disk such that $w$ satisfies the condition (\ref{con-main}).
Then the polynomials are dense in $\mathcal{D}^p(\disk, wdA)$.
\end{theorem}
\begin{proof}
As for approximation, it is enough to work with the following semi-norm (the constant term does not play any role in approximation):
$$\|f\|^p_{\mathcal{D}^p_w}=\int_{\disk}|f^\prime(z)|^pw(z)dA(z).$$
We also note that
$$\|f_r -f\|^p_{\mathcal{D}^p_w}=\int_{\disk}|f_r^\prime(z)-f^\prime(z)|^p w(z)dA(z).$$
Making a change of variables, we obtain
\begin{align*}\int_\disk |f^\prime_r(z)|^pw(z)dA(z)&=
r^{p-k-2}\int_{r\disk} |f^\prime(z)|^p r^kw\left(\frac{z}{r} \right)dA(z)\\
&\le Cr^{p-k-2}\int_{r\disk} |f^\prime(z)|^p w(z)dA(z)\\
&\le C_{p}\int_{r\disk} |f^\prime(z)|^p w(z)dA(z),
\end{align*}
where $C_{p}$ is a constant depending on $p$. Therefore by the dominated convergence theorem,
$$\limsup_{r\to 1}\int_\disk |f^\prime_r(z)|^pw(z)dA(z)\le \int_\disk |f^\prime(z)|^pw(z)dA(z).$$
The rest of the proof goes in the same way as in the proof of Theorem \ref{limsup}; just replace $f$ and $f_r$, by $f^\prime$ and $f^\prime_r$, respectively, to obtain
$$\lim_{r\to 1}\|f_r -f\|^p_{\mathcal{D}^p_w}=\lim_{r\to 1}\int_{\disk}|f_r^\prime(z)-f^\prime(z)|^p w(z)dA(z)=0.$$
\end{proof}

The above theorem extends a recent result of the author (for details, see \cite{abkar3}).
We now state a similar statement for the analytic Besov spaces.
\begin{theorem}\label{limsup-dirichlet}
Let $2<p<\infty$ and $d\mu(z)=w(z)dA(z)$ be a finite positive measure on the unit disk
such that $w$ satisfies the condition (\ref{con-main}).
Then the polynomials are dense in $\mathcal{B}^p(\disk, wdA)$.
\end{theorem}
\begin{proof}
It suffices to verify that for each
$f\in \mathcal{B}^p(\disk, d\mu)$ we have
$$\limsup_{r\to 1^{-}}\int_\mathbb{D}(1-|z|^2)^{p-2}|f^\prime_r(z)|^pd\mu(z)\le \int_\mathbb{D}(1-|z|^2)^{p-2}|f^\prime(z)|^pd\mu(z).$$
The rest of the proof goes in the same way as in the proof of Theorem \ref{limsup}. To this end, we work with
$$\|f\|^p_{\mathcal{B}^p_w}=\int_{\disk}(1-|z|^2)^{p-2}|f^\prime(z)|^pw(z)dA(z).$$
Therefore, by a change of variables we have
\begin{align*}\|f^\prime_r\|^p_{\mathcal{B}^p_w}&=\int_{\disk}(1-|z|^2)^{p-2}|f_r^\prime(z)|^pw(z)dA(z)\\
&= r^{p-k-2}\int_{r\disk}\left (\frac{r^2-|z|^2}{r^2}\right )^{p-2}|f^\prime(z)|^pr^kw\left (\frac{z}{r}\right )dA(z)\\
&\le C_{p}\int_{r\disk}\left (\frac{r^2-|z|^2}{r^2}\right )^{p-2}|f^\prime(z)|^pw(z)dA(z),
\end{align*}
where $C_p$ is a constant depending on $p$. But the function
$$r\mapsto \left (\frac{r^2-|z|^2}{r^2}\right )^{p-2}$$
is increasing in $r$ when $p\ge 2$. Hence the monotone convergence theorem can be applied to the last integral, and then another application of the dominated convergence theorem yields
$$\limsup_{r\to 1^{-}}\int_\mathbb{D}(1-|z|^2)^{p-2}|f^\prime_r(z)|^pd\mu(z)\le \int_\mathbb{D}(1-|z|^2)^{p-2}|f^\prime(z)|^pd\mu(z).$$
As we already mentioned, the rest of the proof is routine.
\end{proof}

The preceding theorem generalizes a result already announced by the author in \cite{abkar4} where we proved that if $w$ is super-biharmonic,
and satisfies the growth condition (\ref{cond-weight}), then the polynomials are dense in analytic Besov spaces.

\noindent{\bf Angular weights}\\
In contrast to radial weights, let us assume that the weight function depends only on the argument of $z$; that is,
$w(re^{i\theta})=w(\theta)$. We may call such weights {\it angular weights}. For example,
$$w(z)=w(re^{i\theta})=(4\pi^2-\theta^2)^\alpha,\quad 0\le \theta <2\pi , \, \alpha>0,$$
is an angular weight that satisfies
$$\int_0^{2\pi}w(\theta)d\theta <\infty.$$
It seems that the study of angular weights was overlooked in the literature. Here we provide some statements on the approximation by polynomials in this weighted spaces.
\begin{proposition}\label{angular}
Let $w$ be an angular weight function satisfying
$$\int_0^{2\pi}w(\theta)d\theta <\infty.$$
Then the Taylor polynomials of each function in $A^2(\disk, wdA)$ tend to the same function in norm. In particular,
the polynomials are dense in $A^2(\disk, wdA)$.
\end{proposition}
\begin{proof}
Let $f(z)=\sum_{n=0}^\infty a_nz^n$ be a function in $A^2(\disk, wdA)$. Note that
\begin{align*}\|f\|^2_{A^2(\disk, wdA)}&=\int_0^1\int_0^{2\pi}\sum_{n=0}^{\infty}|a_n|^2r^{2n+1}w(\theta)drd\theta\\
&=\sum_{n=0}^{\infty}\frac{|a_n|^2}{2(n+1)}\int_0^{2\pi}w(\theta)d\theta.\end{align*}
This implies that for $p_k(z)=\sum_{n=0}^{k}a_nz^n$ we have
$$\|f-p_k \|^2_{A^2(\disk, wdA)}=\left(\int_0^{2\pi}w(\theta)d\theta\right)\sum_{n=k+1}^{\infty}\frac{|a_n|^2}{2(n+1)}.$$
Now, we let $k$ tend to infinity, so that $p_k\to f$ in the norm.
\end{proof}

This result can be generalized in the following way.
\begin{theorem}\label{angular-main}
Let $w$ be an angular weight function satisfying
$$\int_0^{2\pi}w(\theta)d\theta <\infty.$$
Then the polynomials are dense in $A^p(\disk, wdA),\, 0<p<\infty$.
\end{theorem}
\begin{proof}
Since $w$ is angular, we see that
\begin{align*}\|f_r\|^p_{A^p(\disk, wdA)}&=\int_{\disk}|f_r(z)|^pw(z)dA(z)\\
&= \frac{1}{r^2}\int_{r\disk}|f(z)|^p w(\frac{z}{r})dA(z)\\
&=\frac{1}{r^2}\int_{r\disk}|f(z)|^p w(z)dA(z).
\end{align*}
Therefore,
$$\limsup_{r\to 1^-}\int_{\disk}|f_r(z)|^pw(\theta)dA(z)= \int_{\disk}|f(z)|^pw(z)dA(z),$$
from which the result follows.
\end{proof}

This theorem can be generalized to the Dirichlet spaces, as well as to the analytic Besov spaces (we skip the details). Instead, we consider weights that are multiples of a radial weight and an angular weight.
\begin{proposition}\label{radial-angular}
Let $w(re^{i\theta})=\omega(r)v(\theta)$ where $\omega$ and $v$ are two weights satisfying
$$\int_0^{1}r\omega(r)dr <\infty,\quad \int_0^{2\pi}v(\theta)d\theta <\infty.$$
Then the Taylor polynomials of each function in $A^2(\disk, wdA)$ tend to the same function in norm. In particular, the polynomials are dense
in $A^2(\disk, wdA)$.
\end{proposition}
\begin{proof}
As in the proof of the Proposition \ref{angular}, we have
\begin{align*}\|f\|^2_{A^2(\disk, wdA)}&=\int_0^1\int_0^{2\pi}\sum_{n=0}^{\infty}|a_n|^2r^{2n+1}\omega(r)v(\theta)drd\theta\\
&=\left( \int_0^{2\pi}v(\theta)d\theta\right )\sum_{n=0}^{\infty}|a_n|^2 \omega_n,\end{align*}
where
$$\omega_n=\int_{0}^{1}r^{2n+1}\omega(r)dr\le\int_{0}^{1}r\omega(r)dr<\infty.$$
As before, we have
$$\|f-p_k \|^2_{A^2(\disk, wdA)}=\left( \int_0^{2\pi}v(\theta)d\theta\right )\sum_{n=k+1}^{\infty}|a_n|^2\omega_n,$$
which tends to zero as $k\to\infty$, since the sequence $\omega_n$ is bounded by $\omega_0$.
\end{proof}
\begin{theorem}\label{radial-angular-main}
Let $w(se^{i\theta})=\omega(s)v(\theta)$ where $\omega$ and $v$ are two weights satisfying
$$\int_0^{1}s\omega(s)ds <\infty,\quad \int_0^{2\pi}v(\theta)d\theta <\infty,$$
and $r^k\omega(s/r)\le C\omega(s)$ for some integer $k\ge 0$.
Then the polynomials are dense in the weighted Bergman space $A^p(\disk, wdA),\, 0<p<\infty$.
\end{theorem}
\begin{proof}
Again, we see that for $z=se^{i\theta}$,
\begin{align*}\|f_r\|^p_{A^p(\disk, wdA)}&=\int_{\disk}|f_r(z)|^p\omega(s)v(\theta)dA(z)\\
&= \frac{1}{r^{k+2}}\int_{r\disk}|f(z)|^p r^k\omega(s/r)v(\theta)dA(z)\\
&\le \frac{C}{r^{k+2}}\int_{r\disk}|f(z)|^p \omega(s)v(\theta)dA(z).\\
\end{align*}
Therefore,
$$\limsup_{r\to 1^-}\int_{\disk}|f_r(z)|^p\omega(s)v(\theta)dA(z)\le \int_{\disk}|f(z)|^p\omega(s)v(\theta)dA(z),$$
from which the result follows.
\end{proof}

This result holds for the Dirichlet spaces and for the analytic Besov spaces. We omit the details.

\section {Carath\'{e}odory domains}
Let $G$ be a bounded simply connected Jordan domain in the complex plane. This means that $G$ is a bounded region whose boundary is a simple Jordan curve $C$. This curve divides the extended complex plane into a bounded region (the interior of curve) and an unbounded region (the exterior of curve). $C$ is the joint boundary of this two regions. \par
Let $w:G\to [0,\infty)$ be a continuous function with the property that
$$\int_G w(z)dA(z)<\infty.$$
Again, we set $d\mu(z)=w(z)dA(z)$ where $dA(z)$ is the usual area measure on $G$. A function $f$ analytic on $G$ belongs to the Bergman space $A^p(G,d\mu)$ if
$$\int_G |f(z)|^p d\mu(z)<\infty.$$
The above assumption on $w$ implies that the constant function $1$ and all the analytic polynomials belong to $A^p(G,d\mu)$. The question is under what conditions on $w$ can we find polynomials to approximate a given function in $A^p(G,d\mu)$. Note that, due to the lack of symmetry in the domain $G$, the dilated functions cannot be used. We will use the method due to Torsten Carleman \cite{car} to find the approximating polynomials. This method was already employed by
O. J. Farrell in \cite{far1} and \cite{far2} to prove the density of polynomials in the Bergman spaces $A^p(G,dA)$ where the weight function was identically $1$.
\begin{theorem}\label{region}
Let $G$ be a bounded simply connected Jordan domain in the complex plane, and $w$ be a weight function on $G$ such that $\int_G w(z)dA(z)$ is finite.
Let $0<p<\infty$, and $f\in A^p(G, wdA)$. Then there is a sequence of polynomials $p_n$ such that $p_n\to f$ uniformly on compact subsets of $G$, and
$$\lim_{n\to\infty}\int_G |p_n(z)-f(z)|^pw(z)dA(z)=0.$$
\end{theorem}
\begin{proof}
Let $d\mu(z)=w(z)dA(z)$, and
let $G_1\supset G_2\supset G_3\supset \cdots \supset \overline{G}$ be a decreasing sequence of simply connected bounded domains such that $G=\cap_{n=1}^\infty G_n$. Let $\psi_n:G_n\to G$ be the conformal mapping that keeps a fixed point $z_0\in G$ and a fixed direction at $z_0$ invariant.
It is known that $\psi_n(z)\to z$, and $\psi_n^\prime(z)\to 1$; both sequences converge uniformly on compact subsets of $G$.
Let
$$F_n(z)=f(\psi_n(z))(\psi_n^\prime(z))^{2/p},\quad n\ge 1.$$
Each $F_n$ is analytic in $\overline{G}$, and $F_n\to f$ uniformly on compact subsets of $G$.
Note that each $F_n$ can be uniformly approximated by polynomials in $\overline{G}$, and $F_n\to f$ uniformly on compact subsets of $G$. Now, approximate $f$ by $F_n$ and $F_n$ by a polynomial $p_n$; therefore we can find a sequence of polynomials $p_n$ such that $p_n\to f$ uniformly on compact subsets of $G$.
Since both $|f|^p$ and $|F_n|^p$ (for fixed $n$) have finite integral on $G$, we can choose a compact set $K$ in $G$ such that the integrals of
$|f|^p$ and $|F_n|^p$ with respect to $d\mu$ on $G\setminus K$ are sufficiently small (less than $\epsilon$).
Since $p_n$ approximates $F_n$ uniformly on $\overline{G}$,
$$\left |\int_G |p_n|^pd\mu(z)-\int_G |F_n|^pd\mu(z)\right | <\epsilon \mu(G).$$
Now write
$$\left |\int_G |p_n|^pd\mu(z)-\int_G |f|^pd\mu(z)\right | \le \epsilon \mu(G)+
\left | \int_G |F_n|^pd\mu(z)-\int_G |f|^pd\mu(z)\right |.$$
We now choose a compact set $K$ in such a way that the integrals on its complement are small, so that
$$\left | \int_G |F_n|^pd\mu(z)-\int_G |f|^pd\mu(z)\right |\le 2\epsilon +
\left | \int_K |F_n|^pd\mu(z)-\int_K |f|^pd\mu(z)\right |.$$
Again, we use that the convergence on $K$ is uniform, so that
$$\left | \int_K |F_n|^pd\mu(z)-\int_K |f|^pd\mu(z)\right |\le \epsilon \mu(G).$$
In sum, we obtain
$$\lim_{n\to\infty}\int_G |p_n(z)|^pd\mu(z)=\int_G |f(z)|^pd\mu(z).$$
By an argument similar to the one used in the proof of Theorem \ref{limsup}, we conclude that
$$\lim_{n\to\infty}\int_G |p_n(z)-f(z)|^p d\mu (z)=0.$$
\end{proof}



\end{document}